\documentclass[12pt]{amsart}
\usepackage{amsmath}
\usepackage{amsthm}
\usepackage{amssymb}
\usepackage{amscd}
\usepackage{amsfonts}
\usepackage{amsbsy}
\usepackage{graphicx}
\usepackage{tikz}

\newtheorem{theorem}{Theorem}
\newtheorem{remark}{Remark}
\newtheorem{proposition}{Proposition}
\newtheorem{lemma}{Lemma}
\newtheorem{corollary}{Corollary}
\newtheorem{definition}{Definition}

\def\ie{{\em i.e.,} }
\def\eg{{\em e.g.} }

\newfont\bbf{msbm10 at 12pt}
\def\eps{\varepsilon}
\def\phi{\varphi}
\def\R{{\mathbb R}}

\def\N{{\mathbb N}}
\def\Q{{\mathbb Q}}
\def\Z{{\mathbb Z}}
\renewcommand{\S}{{\mathbb S}}

\def\orb{\mbox{\rm orb}}

\def\theta{\vartheta}

\def\eps{\varepsilon}
\def\itin{\boldsymbol{i}}

\begin{document}

\title{Topological properties of Lorenz maps derived from unimodal maps}

\author{Ana Anu\v{s}i\'c, Henk Bruin, Jernej \v{C}in\v{c}}

\address[A.\ Anu\v{s}i\'c]{Departamento de Matem\'atica Aplicada, IME-USP, Rua de Mat\~ao 1010, 
Cidade Universit\'aria, 05508-090 S\~ao Paulo SP, Brazil}
\email{anaanusic@ime.usp.br}
\address[H.\ Bruin]{Faculty of Mathematics,
University of Vienna,
Oskar-Morgenstern-Platz 1, A-1090 Vienna, Austria}
\email{henk.bruin@univie.ac.at}
\address[J.\ \v{C}in\v{c}]{AGH University of Science and Technology, Faculty of Applied Mathematics, 
al.\ Mickiewicza 30, 30-059 Krak\'ow, Poland. -- and -- National Supercomputing Centre IT4Innovations, 
Division of the University of Ostrava, Institute for Research and Applications of Fuzzy Modeling, 
30. dubna 22, 70103 Ostrava, Czech Republic}
\email{jernej@agh.edu.pl}
\thanks{AA was supported by grant 2018/17585-5, S\~ao Paulo Research Foundation (FAPESP).
HB gratefully acknowledges the support of the FWF stand-alone grant number P31950-N45.
J\v{C} was supported by the FWF Schr\"odinger Fellowship stand-alone project J-4276 and 
University of Ostrava grant IRP201824 “Complex topological structures”.}

\thanks{}
\date{\today}


\begin{abstract} A symmetric Lorenz map is obtain by ``flipping'' one of the two branches of a symmetric unimodal map.
We use this to derive a Sharkovsky-like theorem for symmetric Lorenz maps, and also to find
cases where the unimodal map restricted to the critical omega-limit set is conjugate to a Sturmian shift.
This has connections with properties of unimodal inverse limit spaces embedded as attractors of some planar homeomorphisms.
\end{abstract}

\maketitle 

\section{Introduction}\label{sec:intro}
Topological properties of a continuous dynamical system are, in general, easier to understand than those of discontinuous systems.
 For example, for continuous functions of the real line there is the celebrated Sharkovsky Theorem \cite{Shark64},
 which says that if the map has a periodic point of prime period $n$, it also has a periodic point of prime period $m$ for every $m \prec n$
in the Sharkovsky order
$$
1 \prec 2 \prec 4 \prec 8 \prec \dots \prec 4 \cdot 7 \prec 4 \cdot  5  \prec 4 \cdot 3  
\dots \prec 2 \cdot 7 \prec 2 \cdot  5  \prec 2 \cdot 3 \dots \prec 7 \prec 5 \prec 3.
$$
However, in general there is no analogue of the Sharkovsky theorem for discontinuous functions of the reals. 

In this paper we study Lorenz maps, which are piecewise monotone interval maps with a single discontinuity point.
Such Lorenz maps appear as Poincar\'e maps of geometric models of Lorenz attractors described independently by Guckenheimer \cite{Gu}, 
Williams \cite{Wil4} and Afraimovich, Bykov and Shil'nikov \cite{ABS}.
For the class of  ``old'' maps (discontinuous degree one interval maps) which also include Lorenz maps, a 
characterisation of periodic orbit forcing was given by Alsed\'a, Llibre, Misurewicz and Tresser in \cite{ALMT}. 
Hofbauer in \cite{Hof86} obtained a result similar as in \cite{ALMT} using an oriented graph with infinitely many
vertices whose closed paths represent the periodic orbits of the map except that he did not characterize completely 
the set of periodic points. In \cite{ACS} the connection between $\beta$-expansions and Sharkovsky's ordering is given. Recently, Cosper \cite{Cosper} proved a direct analogue of Sharkovsky's theorem in special 
families of piecewise monotone maps (truncated tent map family).

In the present paper we combine some old and more recent results on the relation between unimodal and Lorenz maps,
including a version of Sharkovsky's Theorem. 
The basic idea is to explore the relation between a unimodal map $f$ and symmetric Lorenz maps $\varphi$ and $\psi$ obtained by 
``flipping the right branch'' and ``flipping the left branch'' of the graph of $f$ respectively, see Figure~\ref{fig:Lorenz}. 

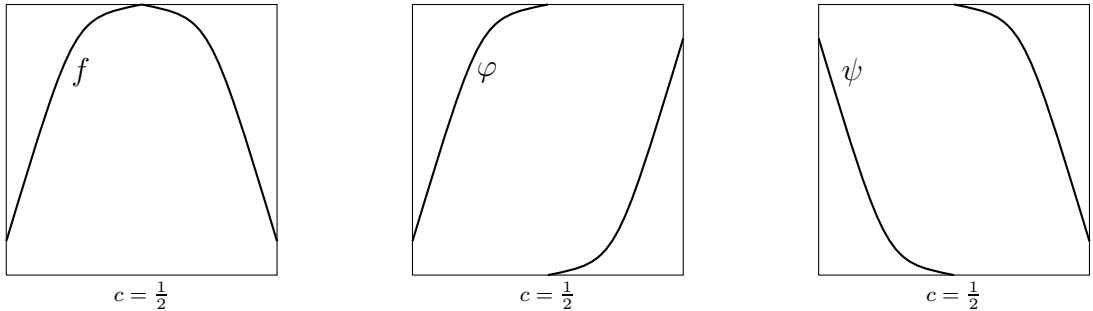
\begin{figure}[ht]
	\begin{center}
		\begin{tikzpicture}[scale=0.9]
		\draw[-, draw=black] (0,0) -- (4,0) -- (4,4) -- (0,4) -- (0,0); 
		\draw[-, draw=black] (6,0) -- (10,0) -- (10,4) -- (6,4) -- (6,0);
		\draw[-, draw=black] (12,0) -- (16,0) -- (16,4) -- (12,4) -- (12,0);
		\draw[-, thick] (0,0.5) .. controls (1, 3.8) .. (2,4); 
		\draw[-, thick] (4,0.5) .. controls (3, 3.8) .. (2,4);
		\draw[-, thick] (6,0.5) .. controls (7, 3.8) .. (8,4); 
		\draw[-, thick] (10,3.5) .. controls (9, 0.2) .. (8,0); 
		\draw[-, thick] (12,3.5) .. controls (13, 0.2) .. (14,0); \draw[-, thick] (16,0.5) .. controls (15, 3.8) .. (14,4); 
		\node at (2,-0.3) {\tiny $c=\frac12$}; \node at (8,-0.3) {\tiny $c=\frac12$};
		\node at (14,-0.3) {\tiny $c=\frac12$};  
		\node at (7.1, 3){$\varphi$}; \node at (1.1, 3) {$f$};
		\node at (12.5, 3) {$\psi$};
		\end{tikzpicture}
		\caption{\label{fig:Lorenz} An increasing and decreasing symmetric Lorenz map $\varphi$ and $\psi$ obtained from a unimodal map $f$.}
	\end{center}
\end{figure} 

For increasing Lorenz maps $\varphi$ we prove in Theorem~\ref{thm:shark} that Sharkovsky's Theorem holds with the exception 
of the fixed points and for decreasing Lorenz maps $\psi$ we prove in Theorem~\ref{thm:shark2} that Sharkovsky's Theorem 
holds possibly except for periods $2^r$, $r \geq 1$.

We can turn $\varphi$ into a proper circle endomorphism (with unique rotation number independent of $x \in \S^1$)
by setting (see also Figure~\ref{fig:barLorenz}):
$$
\bar\varphi(x) = \begin{cases}
\varphi(1)=\widetilde{f(1)}, & x \in [0,a];
\text{ where } a < c \text{ is such that } \varphi(a)=\varphi(1),\\
\varphi(x), & \text{otherwise.}
\end{cases}             
$$
In Proposition~\ref{prop:rotnumber} we calculate the rotation number of the family of such maps, and prove that in the irrational rotation number case the restriction to omega limit set is a minimal homeomorphism. We use techniques developed primarily for unimodal interval maps.

Next, we also give an implementation of Sturmian shifts in interval maps. For every Sturmian shift we assign a unimodal map (basically a kneading sequence) so that the unimodal map restricted to its omega limit set is conjugate to that Sturmian shift.

Maps $\bar\varphi$, besides being interesting on their own, prove also to be very useful in surface dynamics. Namely, knowledge of their dynamics can be related to special orientation preserving planar embeddings of inverse limit spaces with bonding maps being $f$.  
In the last section of the paper we connect the map $\bar{\varphi}$ to the study of unimodal inverse limit spaces represented as attractors of some planar homeomorphisms (this was initially done in \cite{BdCH} using a map conjugated to $\bar{\varphi}$). In Theorem~\ref{thm:accessible} 
we give a compete characterisation of accessible points of tent inverse limit spaces embedded in such a way. Then Corollary~\ref{cor:fp} gives a partial answer to Problem 1 in \cite{AABC} by giving an example of tent inverse limit space which has uncountably many inhomogeneities with only countably infinitely many of them not being endpoints. 

\section{Preliminaries}\label{sec:prelim}

Let $I := [0,1]$ be the unit interval, and $f:I \to I$ a symmetric unimodal map, i.e.,
given the involution $\tilde x = 1-x$, we assume that
$f(\tilde x) = f(x)$ for every $x$.
This means that the critical point $c = \frac12$, and by an appropriate scaling, we can assume that $f(c) = 1$.
For example, $f_a(x) = 1-a(x-\frac12)^2$ with $a \in (0,4]$ is the logistic family in this scaling.

We can turn $f$ into an (increasing) symmetric Lorenz map  $\varphi:I \to I$ by flipping the right half of the graph vertically 
around $c = \frac12$, see Figure~\ref{fig:Lorenz}, giving the following result:
$$
\varphi(x) = \begin{cases}
              f(x) & \text{ if } x \in [0,c],\\
              \widetilde{f(x)} & \text{ if } x \in (c,1].
             \end{cases}
$$
The choice $\varphi(c) = f(c) = 1$ is arbitrary, only made to be definite.

Then, $\varphi$ is semi-conjugate to $f$: $f\circ \varphi = f \circ f$. 
In fact
\begin{equation}\label{eq:ori}
\varphi^n(x) = \begin{cases}
                 f^n(x) & \text{ if } f^n \text{ is increasing at } x;\\
                 \widetilde{f^n(x)} & \text{ if } f^n \text{ is decreasing at } x.
               \end{cases}
\end{equation}
We can also flip the left branch of $f$ and obtain $\psi := \tilde \varphi$
which is called a {\em decreasing symmetric Lorenz map}.
Then $\widetilde {\psi(x)} = \psi(\tilde x)$ for all $x$, and by induction
$$
\psi^n(x) = \begin{cases}
             \tilde \varphi^n(x) = \widetilde{\varphi^n(x)} & \text{ if $n$ is odd,}\\
             \varphi^n(x)  & \text{ if $n$ is even.}\\
            \end{cases}
$$
Suppose then $\psi^n$ is continuous at $x$.
Then \eqref{eq:ori} implies that 
$$
\psi^n(x) = f^n(x) \text{ if and only if }
\begin{cases} 
f^n \text{ is decreasing at $x$} & \text{ and $n$ is odd,}\\
f^n \text{ is increasing at $x$} & \text{ and $n$ is even,}
\end{cases}
$$
and $\psi^n(x) = \widetilde{f^n(x)}$ otherwise.

\section{Sharkovsky's Theorem for Lorenz maps}\label{sec:Shark}

We can describe the dynamics of $f$ using the standard symbolic dynamics with the alphabet $\{ 0, *, 1\}$, where
the symbols stand for the sets $[0,c)$, $\{ c \}$ and $(c,1]$ respectively. It is also enough to restrict the study 
to the dynamical core $[f(0),1]$, since points from $[0,f(0))$ will be mapped to the core under $f$. 
The kneading invariant $\nu\in\{0,*,1\}^{\N}$ is the itinerary of the point $1=f(c)$. 
Since the itinerary map $x \mapsto \itin(x)$ is monotone in the parity-lexicographical order on $\{ 0, *, 1\}^\N$,
the kneading invariant is maximal admissible sequence, i.e.,
$\itin(x) \leq_{pl} \nu$ for all $x \in I$. Also, $\itin(x)\geq_{pl}\sigma(\nu)$ for all $x\in[f(0),1]$. 
It can be shown that every itinerary for which every shift is in parity-lexicographical ordering between $\sigma(\nu)$ 
and $\nu$ can be realized by a point in the dynamical core (see \eg \cite{MT}).

Also, if an $m$-periodic point $y$ is closest to $c$ from all the points in its orbit, 
and $\itin(f(y)) = \overline{e_1 \dots e_m}$, then $\sigma^n(\overline{e_1 \dots e_m}) \leq_{pl} \overline{e_1 \dots e_m}$ for all $n\geq 1$.
As a corollary, $\overline{e_1 \dots e_{m-1}e'_m}$ (if admissible) is periodic of period $k = m$ or $k = m/2$, 
which we prove in the rest of this paragraph. To prove that, assume that there is $k\geq 3$ such that for $j=m/k$ 
we can write $e_1 \dots e_{m-1}e'_m=(e_1\dots e_j)^k$. Then,since $\overline{e_1 \dots e_m}$ is maximal among its shifts, 
$e_1\dots e'_j <_{pl} e_1\dots e_j$, and thus $\#_1(e_1\dots e_j)$ is odd. But then $e_1\dots e_je_1\dots e'_j>_{pl} (e_1\dots e_j)^2$, 
so $\sigma^{(k-2)j}(\overline{e_1 \dots e_{m-1}e_m}) >_{pl} \overline{e_1 \dots e_{m-1}e_m}$,
violating the parity-lexicographical shift-maximality of $\overline{e_1 \dots e_m}$.

\begin{lemma}\label{lem:forcing}
 Let $f$ be a unimodal map with a periodic point $x$ of period $n$. Then for every $m \prec n$, $m>1$,
 there are periodic points $y$ and $y'$ of $f$ such that
 \begin{itemize}
  \item[(a)] $y$ has prime period $m$ and $f^m$ is decreasing at $y$, and
  \item[(b)] $y'$ has prime period $m$ and $f^m$ is increasing at $y'$ or
  $y'$ has prime period $m/2$ and $f^{m/2}$ is decreasing at $y'$.
 \end{itemize}
If $f^n$ is decreasing at $x$, then the statement holds for $m=n$ as well.
\end{lemma}

\begin{proof}
 By Sharkovsky's Theorem, $f$ has at least one periodic orbit of period $m$.
 Take the $m$-periodic point $y$ closest to $c$, so the itinerary $e := \itin(f(y))$
 is maximal (w.r.t.\ the party-lexicographical order $\leq_{pl}$) among all admissible $m$-periodic itineraries.
 Find $e'$ by setting $e'_i = 1-e_i$ if $e_i \neq *$ and $i = km$, $k \in \N$.
 Otherwise we set $e'_i = e_i$.
 Let us first show that $e'$ is admissible.
 Let $j \geq 1$ be the smallest integer such that $e_j \neq \nu_j$.
 If  $j  < m$,  then both $e, e' <_{pl}\nu$. 
 If $m \leq j$ and $\#\{ 1 \leq i \leq m : e_i = 1\}$ is odd, then $e' <_{pl} e <_{pl} \nu$.
 The remaining case is $\#\{ 1 \leq i \leq m : e_i = 1\}$ is even and $m\leq j$.\\
 Assume that $m=j$. Thus $e=\overline{\nu_1\ldots\nu'_m}$. To show that $e'=\overline{\nu_1\ldots\nu_m}$ is admissible, 
 assume that $e'>_{pl}\nu$. Since $\#\{1\leq i\leq m: \nu_i=1\}$ is odd, we have $\nu<_{pl}\sigma^{m}(e')=e'<_{pl}\sigma^m(\nu)$, 
 which contradicts shift-maximality of $\nu$. Thus, $e'$ is admissible in this case. Also, $e'$ cannot be periodic of period $m/2$ 
 since $e_1\ldots e_{m-1}e'_m$ has an odd number of ones. It follows that $e<_{pl}e'$, which contradicts the assumption that $e$ 
 is the closest to $\nu$ among $m$-periodic itineraries, so this case is not possible.\\
 Assume that $m<j$. Then $\sigma^m(e) \leq_{pl} \sigma^m(\nu)$ but since the first symbol at which $\sigma^m(e) = e$ and $\sigma^m(\nu)$
differ is $j-m$, the parity argument and $m$-periodicity of $e$ imply that $\sigma^m(\nu) >_{pl} \nu$, which contradicts the shift-maximality of $\nu$.
So this case cannot occur either.
 
 We conclude that $e'<_{pl}\nu$, and since it is shift-maximal, $\sigma^n(e')<_{pl}\nu$ for every $n\geq 0$. 
 We still have to argue that $\sigma^n(e')>_{pl}\sigma(\nu)$ for every $n\geq 0$. Assume there is $n\geq 0$ such 
 that $\sigma^n(e')<_{pl}\sigma(\nu)$ and take the smallest such $n$. Since $m>1$, $e'$ starts with $1$, and thus $n>0$. 
 Also, since $n$ is the smallest such integer, $\sigma^{n-1}(e')=1\sigma^n(e')$. Then $\sigma^n(e')<_{pl}\sigma(\nu)$ 
 implies that $\sigma^{n-1}(e')=1\sigma^n(e')>_{pl}\nu$, which is a contradiction. We conclude that $e'$ is admissible, 
 \ie realized by a point $y'$ in $[f(0),1]$.
 
 Moreover, we also conclude that $f^m$ is decreasing in $y$ and increasing in $y'$. From the discussion preceding the statement of the lemma, 
 we conclude that the prime period of $y'$ is $m$ or $m/2$.
 
 If $y'$ has prime period $m/2=k$, then $e=\overline{e_1\ldots e_ke_1\ldots e_{k-1}e'_k}$ and $e'=\overline{e_1 \dots e_k}$. Since $\sigma^k(e)<_{pl} e$, 
 we conclude that $\#\{1\leq i\leq k: e_i=1\}$ is odd, from which is follows that $f^k$ is decreasing in $y'$.
\end{proof}

For discontinuous interval maps, there are previous results regarding the forcing relation between 
periods, see e.g.\ \cite{ALMT} which however do not give the following result.

\begin{theorem}\label{thm:shark}
Symmetric increasing Lorenz maps $\varphi$ satisfy Sharkovsky's Theorem, except for the fixed points.
\end{theorem}

\begin{proof} We start the proof for the symmetric Lorenz map $\varphi$ with two claims.
\begin{enumerate}
 \item 
We first show that if $\varphi$ has a periodic point of prime period $n \geq  1$, 
then $f$ also has a periodic point of prime period $n$, unless, possibly,
$n$ is  a power of $2$, and then $f$ has a periodic point of prime period $n$ or $\frac{n}{2}$.

Let $\varphi^n(x)=x$ and 
assume $\varphi^k(x)\neq x$ for all $k<n$. Then the same holds for $\tilde x$.
At exactly one of $x$ and $\tilde x$, say at $x$,
$f^n$ is increasing, so $f^n(x)=x$. 
Assume $k<n$ is such that $f^k(x)=x$ and take the smallest such (so that $x$ has prime period $k$).
If $k$ is not a power of two, then, since $k$ divides $n$, Sharkovsky's Theorem gives the existence of
a periodic point of prime period $n$ as well. So we only have to consider the case
that $k=2^r$.

If $f^k$ is increasing at $x$, then $\varphi^k(x)=f^k(x)=x$, 
a contradiction. Thus $f^k$ must be decreasing at $x$. In that case $f^{2k}$ is increasing at $x$, 
and thus $\varphi^{2k}(x)=f^{2k}(x)=x$, from which we conclude that $n=2k=2^{r+1}$ and $x$ is a periodic point of $f$ of prime period $k=\frac n2$. 
\item
Next we show that if $m>1$ is such that $f$ has an $m$-periodic point, then there exists an $m$-periodic point of $\varphi$. 
Assume $f^m(x)=x$ and $f^k(x)\neq x$ for all $k<m$. 
If $f^m$ is increasing at $x$, then $\varphi^m(x)=x$. 
Assume that there is $k<m$ such that $\varphi^k(x)=x$. Then $f^k$ must be decreasing at $x$, 
and we get $f^{2k}(x)=\varphi^{2k}(x)=x$, thus $m=2k$. 
Now $f^k(\tilde x)=f^k(x)=\widetilde{\varphi^k(x)}=\tilde x$, so $f^{2k}(x)=f^k(\tilde x)=\tilde x$, 
but on the other hand $f^{2k}(x)=f^m(x)=x$, which gives a contradiction. 

The remaining case is when $f^m$ is decreasing at $x$. 
By Lemma~\ref{lem:forcing} and its proof, we find a point $x'$ such that $f^m(x')=x'$ and $f^m$ is increasing at $x'$.
If $m$ is indeed the prime period of $x'$, then we can use the above argument to conclude that $x'$ 
is $m$-periodic point of $\varphi$.
Otherwise, the prime period of $x'$ is $m/2$ and $f^{m/2}$ is decreasing in $x'$. 
But then $\varphi^{m/2}(x') = \widetilde{x}'$ and $\varphi^m(x') = x'$, so $x'$ is periodic for $\varphi$ with prime period $m$.

However, if $m=1$, then $e = \overline 1$, $e' =\overline{0}$
and $x'$ lies in general outside the core
(and in fact outside $I$), so it is lost in the construction of $\varphi$.
Indeed, $\varphi$ has a fixed point only if it comes from a ``full'' unimodal map $f$ (\ie a unimodal map 
that exhibits all possible itineraries of points, such as \eg the quadratic Chebyshev polynomial $f(x) = 4x(1-x)$).
\end{enumerate}
To finish the proof, assume that $\varphi$ has an $n$-periodic point. 
By the first part of the proof, there exists an $n$-periodic point for $f$
(or possibly an $n/2$-periodic point if $n$ is a power of $2$).
Sharkovsky's Theorem implies that $f$ has an $m$-periodic point for every $m \prec n$. 
The second part of the proof implies that there exists an $m$-periodic point of $\varphi$ 
provided $m \neq 1$.
\end{proof}

There are maps $f$ with periodic points of period $2^r$ and no other periods.
If $r$ is maximal with this property, we say that $f$ is of {\em type} $2^r$. If $f$ has periodic points of all periods 
of the form $2^r$ we say that $f$ is of {\em type} $2^\infty$. The union of these two is
called {\em type} $\preceq 2^\infty$.
If $x$ is a $2^r$-periodic point of a unimodal map $f$ of type $\preceq 2^\infty$,
then we say that $x$ has the {\em pattern from the first period doubling cascade}; itinerary of such point is the (shift of the) 
$2^r$-periodic continuation of the Feigenbaum itinerary $\nu^F = \nu_1^F \nu_2^F \nu_3^F \dots$ which equals 
\begin{equation}\label{eq:feig}
1.0.11.1010. 10111011. 1011101010111010.1011101010111011101110101
\dots
\end{equation}
where the dots indicate the powers of $2$.

Studying the decreasing symmetric Lorenz maps $\psi$ we can obtain a theorem similar to Theorem~\ref{thm:shark}. 

\begin{theorem}\label{thm:shark2}
Decreasing symmetric Lorenz maps $\psi$ satisfy Sharkovsky's Theorem,
possibly except for periods $2^r$, $r \geq 1$.
\end{theorem}

\begin{proof}
The proof for a decreasing symmetric Lorenz map $\psi$ is similar as for increasing Lorenz maps.
We only need to repeat the two claims. 
\begin{enumerate}
\item 
Let $\psi^n(x)=x$ and 
assume $\psi^k(x)\neq x$ for all $k<n$. Then the same holds for $\tilde x$.
For even $n$ the proof is the same as for $\varphi$ in Theorem~\ref{thm:shark}, so assume that $n$ is odd.
At exactly one of $x$ and $\tilde x$, say at $x$,
$f^n$ is decreasing, so $f^n(x)=x$. 
Assume $k$ is a divisor of $n$ is such that $f^k(x)=x$. 
Then $k$ and $n/k$ are odd and $f^k$ is decreasing as well, so $\psi^k(x) = x$, which is a contradiction.
\item
Assume that $f$ has a $n$-periodic point and  take $m \prec n$.
(We note that the claim does not hold for $m = n$. Indeed, if $n=m=3$ and the $3$-periodic point is emerging in a saddle node bifurcation,
then $\psi$ does not yet have a $3$-periodic point.)
By Lemma~\ref{lem:forcing}, $f$ has periodic points $x$ of prime period $m$, and if $m$ is not a power of two,
then we can take $x$ orientation preserving as well as orientation reversing. 
Assume that $f^k(x)\neq x$ for all proper divisors $k$ of $m$. 
\begin{itemize}
\item If $m$ is odd, we take $x$ orientation reversing, so that $\psi^m(x) = x$.
Suppose that $j$ is a proper divisor of $m$ such that $\psi^j(x) = x$.
Then $f^j(x) = \tilde x$ because $f^j(x) \neq x$ by assumption. Also, $x=\psi^{m-j}(\psi^j(x))=\psi^{m-j}(x)$ so we also conclude 
that $f^{m-j}(x)=\tilde x$. But then $x=f^m(x) = f^m(\tilde x) = f^{m-j}(f^j(\tilde x))
= f^{m-j}(\tilde x) = \tilde x$, a contradiction.
Therefore $m$ is the prime period of $x$ for $\psi$.

\item If $m$ is even, we take $x$ orientation preserving, so that $\psi^m(x) = x$.
Analogously as above we prove that $m$ is the prime period of $x$ for $\psi$.
\end{itemize}
\end{enumerate}
This shows that $\psi$ satisfies Sharkovsky's Theorem with the potential exception of
periodic points in the first period doubling cascade.
For instance, if $f_a(x) = 1-a(x-\frac12)^2$ with $4 > a > a_{\text{\tiny Feig}}$  (where $a_{\text{\tiny Feig}}$ is the Feigenbaum parameter, 
then $\psi$ does not have a point of prime period $2$, despite the fact that it has periods $n \succ 2$.
More generally, if $f_a$ is $r-1$ renormalisable of period $2$ (so in contrast with Theorem~\ref{thm:shark} the final renormalisation has period $2^{r-1}$),
then $\psi$ has no periodic point of period $2^r$.
The map $\psi$ always has a fixed point, so we don't need to make exceptions for fixed points.
\end{proof}

\section{Cutting times}
 We recall some notation from Hofbauer towers and kneading maps that we use later in the paper; 
for more information on these topics, see e.g.\ \cite[Chapter 6]{BB04}.  

Recall that $c$ denotes the critical point $1/2$. For $n\in\N$ denote by $c_n:=f^n(c)$. We assume that $c_2<c$ (otherwise the dynamics of $f$ is trivial).

Define inductively $D_1 := [c, c_1]$, and 
$$
D_{n+1} := \begin{cases}
           [c_{n+1}, c_1] & \text{ if } c \in D_n;\\
           f(D_n) & \text{ if } c \notin D_n.
          \end{cases}
$$
We say that $n$ is a {\em cutting time} if $c \in D_n$. The cutting times are denoted by
$S_0, S_1, S_2,  \ldots$ (where $S_0=1$ and $S_1=2$). They were introduced
in the late 1970s by Hofbauer \cite{Hof80}.
The difference between consecutive cutting times is again a cutting time (see e.g.\ Subsection 6.1 in \cite{BB04}),
so we can define the kneading map $Q:\N \to \N \cup\{0\}$ as
$$
S_{Q(k)}:=S_k - S_{k-1}.
$$
We call $f$ {\em long-branched } if $\liminf_n |D_n| > 0$, which is equivalent to
$\liminf_k |D_{S_k}| > 0$ and also to $\limsup_k Q(k) < \infty$.

A purely symbolic way of obtaining the cutting times is the following.
Recall that we use the itinerary map $\itin$ for $f$ (and also for $\varphi$) with codes $0$ for $[0,c)$ and $1$ for $(c,1]$.
We will use the modified kneading sequence $\nu=\lim_{x \nearrow c} \itin(x)=10\dots \in \{0,1\}^\N$, where we traditionally omit the zero-th symbol. 
Note that if $c$ is not periodic, $\nu=\itin(c_1)$ and the modification is only made so that the itineraries do not contain symbol $*$ 
(we take the smaller of the two sequences in parity-lexicographical ordering).

We can split any sequence $e \in \{ 0,1\}^\N$
into maximal pieces (up to the last symbol) that coincide with a prefix of $\nu$.
To this end, define
\begin{equation}\label{eq:rho-function}
\rho:\N \to \N, \quad \rho(n) = \max\{k > n : e_{n+1}e_{n+2} \dots e_{n+k-1} \text{ is prefix of } \nu\}.
\end{equation}
That is, the function $\rho$ depends on $e$ and $\nu$, but we will suppress this dependence.
When we apply this for $e = \nu$, we obtain
$$
S_0 = 1, \quad S_{k+1} = \rho(S_k),
$$
or in other words $S_k = \rho^k(1)$ for $e=\nu$ and $k \geq 0$.

Define the {\em closest precritical points}
$\zeta \in I$ as any point such that $f^n(\zeta) = c$ for some $n \geq 1$ and $f^k(x) \neq c$
for all $k \leq n$ and $x \in (\zeta,c)$.
By symmetry, if $\zeta$ is a closest precritical point, 
$\tilde\zeta = 1-\zeta$ is also a closest precritical point.
If $\zeta' \in (\zeta,c)$ is a closest precritical point of the lowest $n' > n$,
then the itineraries of $f(c)$ and $f(x)$, $x \in (\tilde\zeta',\tilde\zeta)$ coincide for exactly
$n'-2$ entries, and differ at entry $n'-1$. Hence $n'$ is a cutting time, say $n' = S_{k'}$ for some $k' \geq 1$.
We use the notation $\zeta = \zeta_{k'}$ if $n' = S_{k'}$.
That is
\begin{equation}\label{eq:ccp}
\dots  < \zeta_k < \zeta_{k+1} < \dots < c < \dots < \tilde \zeta_{k+1} < \tilde \zeta_k < 
\dots \qquad f^{S_k}(\zeta_k) = f^{S_k}(\tilde \zeta_k) = c.
\end{equation}
and 
\begin{equation}\label{eq:Upsilon}
x \in \Upsilon_k := (\zeta_{k-1}, \zeta_k] \cup [\tilde \zeta_k, \tilde \zeta_{k-1}) 
\ \Rightarrow \
\itin(f(x)) = \nu_1 \dots \nu_{S_k-1} \nu'_{S_k}\ldots
\end{equation}
Applying this to $x = f^m(c)$, we obtain that $\rho(m)-m$ is a cutting time. 

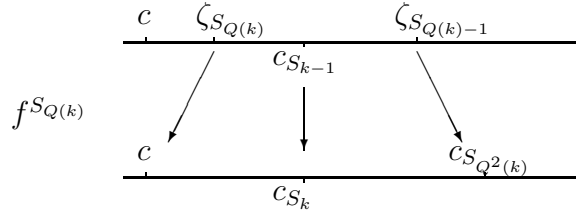
\begin{figure}[htbp]
\begin{center}
\rule{22pt}{0pt}
\begin{minipage}{145mm}
\unitlength=6mm
\begin{picture}(10,5)(-5,-0.5) \let\ts\textstyle
\thicklines
\put(2,0){\line(1,0){10}}
\put(2,3){\line(1,0){10}}
\thinlines
\put(-0.5,1.2){$f^{S_{Q(k)}}$}

\put(2.5,3){\line(0,1){0.1}} \put(2.3,3.4){$c$}
\put(4,3){\line(0,1){0.1}} \put(3.6,3.4){$\zeta_{S_{Q(k)}}$}
\put(6,3){\line(0,-1){0.1}} \put(5.3,2.5){$c_{S_{k-1}}$}
\put(8.5,3){\line(0,1){0.1}} \put(8,3.4){$\zeta_{S_{Q(k)-1}}$}

\put(2.5,0){\line(0,1){0.1}} \put(2.3,0.4){$c$}
\put(6,0){\line(0,-1){0.1}} \put(5.3,-0.5){$c_{S_k}$}
\put(10,0){\line(0,1){0.1}} \put(9.2,0.4){$c_{S_{Q^{ 2}(k)}}$}

\put(4,2.8){\vector(-1,-2){1}}
\put(6,2){\vector(0,-1){1.4}}
\put(8.5,2.8){\vector(1,-2){1}}

\end{picture}
\end{minipage}
\vskip-30pt
\end{center}
\caption{The points $\zeta_{S_{Q(k)}} < c_{S_{k-1}} <\zeta_{S_{Q(k)-1}}$
and their images under $f^{S_{Q(k)}}$.}
\label{Fig:PoscSk}
\end{figure}

In particular,
\begin{equation}\label{eq:zetaQ}
f^{S_{k-1}}(c) \in \Upsilon_{Q(k)} = (\zeta_{ Q(k)-1 }, \zeta_{ Q(k) } ] \cup [ \tilde\zeta_{Q(k)},\tilde \zeta_{Q(k)-1} ),
\end{equation}
see Figure~\ref{Fig:PoscSk},
and the larger $Q(k)$, the closer $f^{S_{k-1}}(c)$ is to $c$.

Let $\kappa = \min\{ j > 1 : \nu_j = 1\}$.
Then we can define the {\em co-cutting times} as 
$$
\hat S_0 = \kappa, \quad \hat S_{k+1} = \rho(\hat S_k),
$$
The cutting and co-cutting times are always disjoint sequences (see \cite[Lemma 2]{bruin}), and $\{ \hat S_k\} = \emptyset$ if 
$f$ is the full unimodal map (because then $\nu = 10000\dots$ and $\kappa$ is not defined).
Furthermore, there is a {\em co-kneading map} $\hat Q:\N \to \N\cup\{0\}$ such that
$$
\hat S_k = \hat S_{k-1} + S_{\hat Q(k)}.
$$

\begin{proposition}\label{prop:Qinfinity}
Let $f$ be a unimodal map with the kneading map $Q$.
 If $Q(k) \to \infty$, then $\hat Q(k) \to \infty$ and $\omega(c)$ is a minimal Cantor set.
\end{proposition}

\begin{proof} In \cite[Lemma 3.6 and Proposition 3.2]{bruinThesis} and \cite[Lemma 4 and Proposition 2]{bruin} 
it was shown that $Q(k) \to \infty$
implies $\hat Q(k) \to \infty$ and that $c$ is persistently recurrent.
This property was introduced by Blokh and implies minimality of $\omega(c)$, see
\cite{BL91} and also \cite[Section 3]{B96}.
\end{proof}

In fact, $\limsup_k Q(k) = \infty$ implies that $\limsup_k \hat Q(k) = \infty$, but not vice versa.
If both $\limsup_k  Q(k) < \infty$ and $\limsup_k \hat Q(k) < \infty$, then $c$ is non-recurrent,
but as we will see in Section~\ref{sec:outer}, there are maps where $\limsup_k  Q(k) < \limsup_k \hat Q(k) = \infty$.

\section{Sturmian shifts}\label{sec:Sturm}
There are multiple ways of defining Sturmian shifts and we take the one using the symbolic dynamics of circle rotations.

\begin{definition}\label{def:rotational}
 Let $R_\alpha:\S^1 \to \S^1$, $x \mapsto x+\alpha \bmod 1$, be the rotation over an irrational angle $\alpha$.
 Let $\beta \in \S^1$ and build the itinerary $u = (u_n)_{n \geq 0}$ by
 \begin{equation}\label{eq:itin_rotation} 
 u_n = \begin{cases}
        1 & \text{ if } R_\alpha^n(x) \in [0,\alpha),\\
        0 & \text{ if } R_\alpha^n(x) \notin [0,\alpha).
       \end{cases}
\end{equation}
Then $u$ is called a rotational sequence. The minimal (and uniquely ergodic) shift space obtained as
$X_\alpha = \overline{\{ \sigma^n(u) : n \in \N\}}$ is the \emph{Sturmian shift} of frequency $\alpha$,
and each $x \in X_\alpha$ is called a Sturmian sequence.
\end{definition}

The purpose of this section is to describe cases when unimodal maps restricted to their critical 
omega-limit sets $\omega(c)$ are conjugate to Sturmian shift.
There are in fact multiple ways of choosing the kneading map $Q$ so that $(\omega(c), f)$ is Sturmian.
The simplest way is by means of the Ostrowski numeration, see \cite{Ostr21}.
Indeed, let $\alpha \in I$ be some irrational number and let $p_n/q_n$ be the convergent of its continued fraction 
expansion.
Thus $q_{-1} = 0$, $q_0=1$ and $q_n = a_n q_{n-1} + q_{n-2}$.
Take $k_n = \sum_{j=0}^n a_j$ and then cutting times as follows:
$$
\begin{cases}
 S_k = k+1 & \text{ for } 0 \leq k \leq a_1,\\
 S_{k_n} = q_n & \text{ for } n \geq 1,\\
 S_{k_n+a} = aq_n + q_{n-1} & \text{ for } 1 \leq a \leq a_n, \ n \geq 1.
\end{cases}
$$
It is clear that $Q(k) \to \infty$ in this case, and the $\{ S_k \}$ interpolate between the numbers $q_n$, see also
\cite{BKS98}. However, $f:\omega(c) \to \omega(c)$ is in general not invertible, 
since $c$ itself and/or other points in the backward orbit of $c$ have two preimages in $\omega(c)$, see also \cite{Br99}.
As such $(\omega(c), f)$ is conjugate to the one-sided Sturmian shift.

However, also when $Q(k)$ is bounded (in fact also when $Q(k) \leq 1$) there
are examples where $(\omega(c), f)$ is Sturmian, see \cite[Chapter III, 3.6]{bruinThesis}.
Let $\varphi:I \to I$ be an increasing symmetric Lorenz map as in previous sections.
In addition to $\itin$,
another way of coding orbits of unimodal maps (used by Milnor \& Thurston \cite{MT},
Collet \& Eckmann \cite{CE} and Derrida et al.\ \cite{DGP}) 
 is as follows: set $\theta_0(x) = +1$ and for $n \geq 1$,
\begin{equation}\label{eq:theta}
 \theta_n(x) = \prod_{j=0}^{n-1} (-1)^{\itin_j(x)} = \begin{cases}
                +1 & \text{ if $f^n$ is increasing at } x;\\
                -1 & \text{ if $f^n$ is decreasing at } x.
               \end{cases}
\end{equation}
It follows that $\theta(f(x)) = \sigma(\theta(x))$ if $\itin_0(x) = 0$ and
$\theta(f(x)) = -\sigma(\theta(x))$ if $\itin_0(x) = 1$.
For the itinerary $\itin^\varphi$ of  $x \in I \setminus \bigcup_{j=0}^n \varphi^{-j}(c)$ under the function $\varphi$
this means that
$$
\itin_n^\varphi(x) = 0 \Leftrightarrow
\begin{cases}
 \itin_n(x) = 0 \text{ and } \theta_n(x) = + 1, \\
 \text{or} & \\
 \itin_n(x) = 1 \text{ and } \theta_n(x) = - 1. 
\end{cases}
\Leftrightarrow
\theta_{n+1}(x) = +1,
$$
and
$$
\itin_n^\varphi(x) = 1 \Leftrightarrow
\begin{cases}
 \itin_n(x) = 1 \text{ and } \theta_n(x) = + 1, \\
 \text{or} & \\
 \itin_n(x) = 0 \text{ and } \theta_n(x) = - 1. 
\end{cases}
\Leftrightarrow
\theta_{n+1}(x) = -1,
$$
In other words, $\itin^\varphi_n = (1-\theta_{n+1}(x))/2$.
This gives $\itin^\varphi \circ \varphi(x) = \sigma \circ \itin^\varphi(x)$.
Also, if $\nu^\varphi = \lim_{x \nearrow c} \itin^\varphi(x)$ with the first symbol neglected,
and defined
$\rho^\varphi(n) = \min\{ k > n : \nu^\varphi_k = \nu^{\varphi}_{k-n}\}$, then we recover the cutting times
as $S_0= 1$, $S_{k+1} = \rho^\varphi(S_k)$.
(The co-cutting times can be recovered as $\hat S_0 = \kappa = \min\{ k \geq 1 : \nu^\varphi_k = 0\}$
and $\hat S_{i+1} = \min\{ k > \hat S_i : \nu^\varphi_k \neq \nu^\varphi_{k-\hat S_i}\}$.) See the example in the proof of Proposition~\ref{prop:rotnumber}.

To each $x \in I$ we can assign a rotation number by first assigning
a lift $\Phi:\R \to \R$ to the Lorenz map $\varphi$:
$$
\Phi(x) = \begin{cases}
           \varphi(x) & \text{ if } x \in [0,c], \ \varphi(c) = 1;\\
           \varphi(x) + 1 & \text{ if } x \in (c,1);\\
           \Phi(x-n)+n & \text{ if } x \in [n,n+1).
          \end{cases}
$$
Then $\Phi(x) \bmod 1 = \varphi(x \bmod 1)$ and the rotation number is defined as
\begin{equation}\label{eq:rot}
\alpha(x) = \limsup_{n\to\infty} \frac{\Phi^n(x)-x}{n},
\end{equation}
Since $\lfloor \Phi(x) \rfloor = \lfloor x \rfloor$ if and only if $x \bmod 1 \in [0,c)$
and $\lfloor \Phi(x) \rfloor = \lfloor x \rfloor + 1$ otherwise, we obtain
\begin{eqnarray}\label{eq:rotrho}
\alpha(x) &=& \limsup_{n\to\infty} \frac1n \#\{ 0 \leq k < n : \itin^\varphi_k(x) = 1 \}  \\
&=& \limsup_{n\to\infty} \frac1n \#\{ 1 \leq k \leq n : \theta_k(x) = -1 \}. \nonumber
\end{eqnarray}

Next we turn $\varphi$ into a proper circle endomorphism (with unique rotation number independent of $x \in \S^1$)
by setting:
$$
\bar\varphi(x) = \begin{cases}
               \varphi(1)=\widetilde{f(1)}, & x \in [0,a];
               \text{ where } a < c \text{ is such that } \varphi(a)=\varphi(1),\\
               \varphi(x), & \text{otherwise.}
              \end{cases}             
$$
Also let $b > c$ be such that $\varphi(b) = a$, see Figure~\ref{fig:barLorenz}.

\begin{figure}[ht]
\begin{center}
\begin{tikzpicture}[scale=0.9]
\draw[-, draw=black] (0,0) -- (4,0) -- (4,4) -- (0,4) -- (0,0) -- (4,4); 
 \draw[-, thick] (0,3.5) -- (1.05, 3.5);
 \draw[-, thick] (1.05,3.5) .. controls (1.4, 3.9) .. (2,4); 
\draw[-] (1.05,0) -- (1.05,0.1); 
\draw[-] (3.2,0) -- (3.2,0.1); 
\draw[dashed] (3.2,0) -- (3.2,1.05)--(1.05,1.05)--(1.05,0); 
 \draw[-, thick] (4,3.5) .. controls (3, 0.2) .. (2,0);
\node at (2,-0.2) {\tiny $c=\frac12$}; \node at (1.05,-0.2) {\tiny $a$}; \node at (-0.35,3.5) {\tiny $\varphi(1)$};  
\node at (1.6, 3.5){$\bar \varphi$}; \node at (3.2,-0.2) {\tiny $b$};
\end{tikzpicture}
\caption{\label{fig:barLorenz} A stunted symmetric Lorenz map $\bar{\varphi}$ as a circle endomorphism.}
\end{center}
\end{figure}
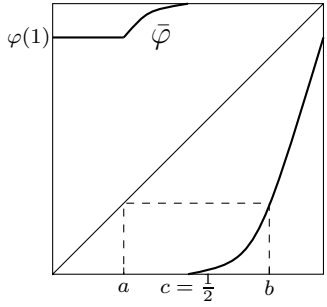 

The circle endomorphism $\bar{\varphi}$ obtained from $\varphi$ was already studied in the last section of \cite{bruinThesis}.

\begin{proposition}\label{prop:rotnumber}
Assume that $f$ is a unimodal map with cutting times $\{ S_j\}_{j \geq 0}$.
Let $b > c$ be such that $\bar{\varphi}(b)=a$, see Figure~\ref{fig:barLorenz}.
Then the rotation number of the corresponding $\bar\varphi$ equals
 $$
 \alpha = \begin{cases}
         \frac{k}{S_k} \in [\frac12, 1] \cap \Q & \text{ if $k$ is minimal such that } 
         f^{S_k}(c) \in (\hat b,b),\\[1mm]
         \lim_{k \to \infty} \frac{k}{S_k} \in [\frac12,1] & \text{ if no such $k$ exists.}
        \end{cases}
 $$
In the latter case, the kneading map $Q(j) \leq 1$ for all $j \in \N$,
and if $\alpha \notin \Q$, then $f:\omega(c) \to \omega(c)$ is a minimal homeomorphism.
\end{proposition}

\begin{proof}
Recall that $f(c) = 1$ and assume that there is a minimal integer $n \geq 1$ such that 
$\varphi^n(1) \in (c,b]$. Then $\bar \varphi^{n+1}(1) \in (0,a]$ and
 $\bar \varphi^{n+2}(1) = \bar\varphi(1)$ is periodic with period $n+1$.

Recall that $b>c$ is such that $\bar\varphi(b) = a$,  so $f(b)=\tilde a>c$, and
$f^2(b) = f(a) = \widetilde{f^2(c)} > c$.
Therefore $b \in (\tilde\zeta_{2}, \tilde\zeta_{1})$
for closest precritical points $\tilde\zeta_1 > \tilde\zeta_2 > c$,
see \eqref{eq:ccp}, and $\tilde b\in(\zeta_1, \zeta_2)$.
There are two possibilities:
\begin{itemize}
 \item $\varphi^n(1) = f^n(1)$. In this case $f^n$ is increasing at $1$ and thus
 $n +1 = S_k$ is a cutting time.
  \item $\varphi^n(1) = \widetilde{f^n(1)}$. In this case $f^n$ is decreasing at $1$ and again
 $n+1 = S_k$ is a cutting time.
\end{itemize}
By minimality of $k$, $f^{S_j}(c) \notin [\tilde b,b] \setminus \{ c \}$ for all $j < k$, and hence
the kneading sequence $\nu$ of $f$ consists of blocks $0$ or $11$. For example:
\begin{eqnarray*}
\nu &=& \ \ 1.\ \ \, \ 0.\ \ \ \ 0.\ \ \ \ 1\ \ \ \, 1.\ \ \ \, 0.\, \ \ \ \, 1\ \ \ \, 1.\ \ \ \, 0.\ \ \, \ \, 1\ \ \ \ \boldsymbol{1}.\ \ \ \ 1\ \ \ \ 0\ \  \ 1 \dots \\
\theta &=& +1\ -1\ -1\ -1 \ +1  -1 \ -1 +1 \ -1 -1 +1 \ -1 \ +1 +1 -1 \dots \\
\nu^\varphi &=& \ \ 1.\ \ \, \ 1.\ \ \ \ 1.\ \ \ \ 0\ \ \ \, 1.\ \ \ \, 1.\, \ \ \ \, 0\ \ \ \, 1.\ \ \ \, 1.\ \ \, \ \, 0\ \ \ \ \boldsymbol{1}.\ \ \ \ 0\ \ \ \ 0\ \  \ 1 \dots \\
\end{eqnarray*}
where dots indicate cutting times and the bold symbol the position $S_k$.
Since $n+1$ is the period of $\bar\varphi(1)$, this shows that $\#\{ 1 \leq j \leq S_k : \theta_j = -1\} = k$, 
and in view of \eqref{eq:rotrho} we have $\alpha = k/S_k$. 

If there is no such minimal $n$, i.e., $\varphi^n(1)\notin (\tilde b,b)$ for all $n\geq 1$,
then $f^{n}(1) \notin (\tilde b,b)$ for all $n \geq 1$ (and in particular $Q(j) \leq 1$)
for all $j \geq 1$.
A counting argument similar to the above shows that 
$\alpha = \limsup_{k\to\infty}k/S_k = \lim_{k\to\infty}k/S_k$.
It is possible that $\alpha$ is rational, e.g., for the logistic map $f_a(x) = 1-a(x-\frac12)^2$ with
$a = 3.5097$. In this case, $\nu = (101)^\infty$ and $\bar\varphi^i(1)$ converges to an attracting orbit of period $3$.
Also for the tent map $T(x) = 1-\lambda|x-\frac12|$ with $\lambda = \frac12(1+\sqrt{5})$, the critical orbit
$\{ \frac12, 1, \frac34-\frac14\sqrt{5}\}$ has period three and avoids $[0,a]$.

If $\alpha \notin \Q$, then $\omega_{\bar\varphi}(c)$ is the Cantor set, disjoint from $[0,a]$ and 
minimal w.r.t.\ the action of $\bar\varphi$.
Under the semi-conjugacy $f$ between $f$ and $\varphi$ (indeed $f \circ f = f \circ \varphi$), this 
projects to a minimal map $f:\omega_f(c) \to \omega_f(c)$.
We will show that $f:\omega_{\varphi}(c) \to \omega_f(c)$ is in fact a homeomorphism, from which it follows that  $f:\omega_{f}(c)\to\omega_f(c)$ is also a homeomorphism.
Assume by contradiction that $x < c < \tilde x$ are points in $\omega_{\varphi}(c)$ 
such that $f(x) = f(\tilde x) = y \in \omega_f(c)$.
Then, since $f$ is the semi-conjugacy between $\varphi$ and $f$, we must have
$f(\varphi^n(\tilde x)) = f(\varphi^n(x)) = f^n(y)$ for every $n\in\N$. Note that $\varphi^n(\tilde x)=\widetilde{\varphi^n(x)}$ 
for every $n\in\N$, and thus $\varphi^n(\tilde x)\neq \varphi^n(x)$, unless $\varphi^n(x)=c$, and thus $f^n(x)=c$. Since $c$ is 
not periodic, there exists $N\in\N$ such that $\varphi^n(\tilde x)\neq \varphi^n(x)$ for all $n\geq N$, and thus $f^n(y)$ has 
two $f$-preimages in $\omega_{\varphi}(c)$.
Since $f\colon\omega_f(c)\to\omega_f(c)$ is minimal, for every $\eps>0$ there exists infinitely many  $y'\in\orb_f(y)$ 
which are $\eps$-close to $f^2(c)=f(1)$. For sufficiently small $\eps$, an $f$-preimage of a point $\eps$-close to $f(1)$ 
will be contained in $[0,a]$. Since every point in $\orb_f(y)$ eventually has both $f$-preimages in $\omega_{\varphi}(c)$, 
we conclude that $\omega_{\bar\varphi}(c)\cap [0,a]=\omega_{\varphi}(c)\cap[0,a]\neq\emptyset$, which is a contradiction. 
\end{proof}

We argued so far that there exist stunted Lorenz maps for which $\overline{\orb_{\bar\varphi}(c)}$
is a Cantor set with dynamics similar to circle rotations (in fact to Denjoy circle maps)
with irrational rotation number, and that there are also unimodal maps with kneading map bounded by $1$, such that
$f|_{\omega(c)}$ is semi-conjugate to a circle rotation, and in fact, the rotation number
is $\alpha = \lim_{k\to\infty}k/S_k$.
Therefore $(\omega(c), f)$ represents a Sturmian shift.

In fact, every irrational rotation number (hence every Sturmian shift) can be realized this way,
as we can prove by studying this rotation number closer.
Indeed, let $\alpha = [0;a_1,a_2,a_3,\dots]$ be the continued fraction expansion of $\rho$,
with convergents $\frac{p_i}{q_i}$.
For the irrational rotation $R_\alpha$, the denominators $q_i$ are the times of
closest returns of any point $x \in \S^1$ to itself, and these returns occur alternatingly on the left and on the right.
If we assume that $R_\alpha^{q_i}(x)$ is to the right of $x$, and set $A_{q_i} = [x,R_\alpha^{q_i}(x)]$,
then the first iterate $k$ such that $R_\alpha^k(A_{q_i}) \owns x$ is $k = q_{i+1}$ and $R_\alpha^{q_{i+1}}(x)$ is to the left of $x$.

For the map $\bar\varphi$, the closest returns on the left indeed accumulate on $c$, but the right neighborhood
$[c,b)$ is the preimage of the plateau $[0,a)$ and no further iterates of $c$ enter that region.
Instead, returns on the left accumulate on $b$.

Translating this back to the unimodal map $f$ with kneading sequence $\nu = \nu_1\nu_2\nu_3\dots$, 
the closest returns on the left correspond to
closest returns at co-cutting times (recall that there are no cutting times
$S_j$ so that $f^{S_j}(c) \in (\tilde b,b)$).
If $q_i$ is such a co-cutting time, then (recalling the function $\rho$ from \eqref{eq:rho-function}
and using the above argument), the Farey convergents\index{Farey convergents}
$\rho^a(q_i) = q_i+aq_{i+1}$ are also the next co-cutting times for $1 \leq a \leq a_{i+1}$,
and in particular, $\rho^{a_{i+1}}(q_i) = q_{i+2}$.

The closest returns on the right correspond to cutting times, but this time $f^{q_i}(c)$
accumulate on $b$, and because $f^3(b) = f^3(c)$, the itinerary of $b$ is 
\begin{equation}\label{eq:ib}
\itin(b) = b_1b_2b_3b_4b_5 \dots = 11\nu_3\nu_4\nu_5 \dots
\end{equation}
Therefore we need to consider the analogous function $\rho_b(m) = \min\{ n > m : b_n \neq b_{n-m}\}$,
and find that $\rho_b^a(q_i) = q_i+aq_{i+1}$  for $1 \leq a \leq a_{i+1}$,
and in particular, $\rho_b^{a_{i+1}}(q_i) = q_{i+2}$.

For example, if $a_i \equiv 2$, so the $q_i$s are the Pell numbers $2,5,12,29,70,189,\dots$, 
then we obtain
$$
\nu = 1\boldsymbol{0}.1'1.\boldsymbol{1}'1.0.11.11.\boldsymbol{0}.11.11.1'1.0.11.11.0.11.11.\boldsymbol{1}'1.0\dots
$$
where dots indicate cutting times and primes co-cutting times. The bold symbols indicate the positions $q_i$.
In fact, for each $i$
$$
 \nu_{q_{i+1}-q_i+1} \dots \nu_{q_{i+1}-1}\nu_{q_{i+1}} =
 \nu_1 \dots \nu_{q_i-1} \nu_{q_i} \text{ or }   \nu_1 \dots \nu_{q_i-1} \nu'_{q_i} \qquad
 \text{ for each even } i \in \N,
$$
and therefore $c$ has two limit itineraries $\lim_{x \nearrow c} \itin(x) = 0\nu$ and $\lim_{x \searrow c} \itin(x) =1\nu$,
but $c$ has only one preimage in $\omega(c)$.

\section{Outside maps and unimodal inverse limit spaces}\label{sec:outer}
Boyland, de Carvalho and Hall in \cite[Section 3]{BdCH} present a different way of creating a circle endomorphism 
from a unimodal map.
They call this the {\em outside map} $B$, and use it to study the inverse limit space of
the unimodal map as attractors of sphere homeomorphisms.
Starting from a unimodal map $f:I \to I$ such that the second branch is surjective
(i.e., $f([c,1]) = I$), they double the interval to a circle $\R / 2\Z = [0,2]/_{0 \sim 2}$, 
and let $B$ map the second branch onto $[1,2]$ by flipping this branch, and then extend the definition of $f$ on $[1, 2-d]$
for the unique point $d \in (c,1]$ for which $f(d)=f(0)$
to cover the interval $[0, f(0)]$. The remaining interval $[2-d,2]$
is then mapped to the constant $f(0)$.
That is
$$
B(x) = \begin{cases}
f(x) & \text{ if } x \in [0,c);\\
2-f(x) & \text{ if } x \in [c,1);\\
f(2-x)& \text{ if } x \in [1,2-d);\\
f(0) & \text{ if } x \in [2-d,2),
\end{cases}
$$
see Figure~\ref{fig:Outside-Lorenz}.
Let us carry this out for the \emph{family of cores of tent maps} $T_{\lambda}\colon I\to I$, 
$$T_{\lambda}=\begin{cases}
{\lambda}x+2-{\lambda}, & x\in[0,\frac{{\lambda}-1}{{\lambda}}],\\
-{\lambda}x+{\lambda}, & x\in[\frac{{\lambda}-1}{{\lambda}}, 1],
\end{cases}$$
for all ${\lambda}\in(1,2]$.

Then the map
\begin{equation}\label{eq:outsidemap}
 \bar\varphi(x) = \begin{cases}
                   \frac{\lambda}{2} & \text{ if } 0 \leq x \leq a = \frac{\lambda-1}{\lambda};\\[2mm]
                   \lambda (x-\frac12) \bmod 1 & \text{ if } a = \frac{\lambda-1}{\lambda} \leq x < 1,
                  \end{cases}
\qquad \text{ on } \R/\Z
\end{equation}
and the outside map
$$
 B(x) = \begin{cases}
                   \lambda (x-1) + 2 \bmod 2 & \text{ if } 0 \leq x < \frac{2}{\lambda}; \\[2mm]
                   2-\lambda & \text{ if } \frac{2}{\lambda} \leq x < 2,\\
                  \end{cases}
\qquad \text{ on } \R/2\Z
$$
are conjugate with conjugacy $G\colon\R/\Z\to \R/2\Z$,  $G(x) = 2(1-x) \bmod 2$,
i.e., $G \circ \bar\varphi = B \circ G$.
But the conjugacy reverses orientation, so the rotation numbers are each others opposite, $\alpha$ for $\bar\varphi$
versus $1-\alpha$ for $B$.

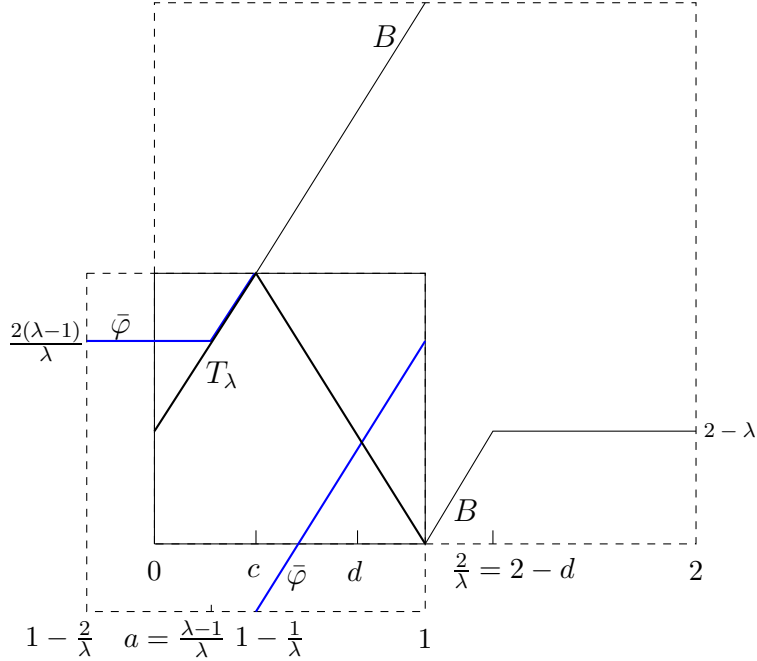
\begin{figure}[ht]
\begin{center}
\begin{tikzpicture}[scale=0.9]
\draw[-, draw=black] (1,1) -- (1,5) -- (5,5) -- (5,1) -- (1,1); 
\draw[-, dashed] (0,5) -- (5,5) -- (5,0) -- (0,0) -- (0,5);
\draw[-, dashed] (1,1) -- (1,9) -- (9,9) -- (9,1) -- (1,1);
\draw[-, draw=blue, thick] (0,4) -- (1.83,4) -- (2.48,5);
\draw[-, draw=blue, thick] (2.5,0) -- (5,4);
\draw[-, draw=black] (5,1) -- (6,2.6666) -- (9,2.6666);
\draw[-, draw=black] (2.5,5)--(5,9);
\draw[-, thick] (1,2.666) -- (2.5,5) -- (5,1);
\node at (9,0.6) {\small $2$}; 
\node at (1, 0.6) {\small $0$}; 
\node at (1.3,-0.4) {\small $a = \frac{\lambda-1}{\lambda}$};\draw[-, draw=black] (1.84,0) -- (1.84,0.1);
\node at (6.3,0.6) {\small $\frac{2}{\lambda}=2-d$}; \draw[-, draw=black] (6,1) -- (6,1.2); 
\node at (3.96,0.6) {\small $d$}; \draw[-, draw=black] (4,1) -- (4,1.2); 
\node at (2.46,0.6) {\small $c$}; \draw[-, draw=black] (2.5,1) -- (2.5,1.2); 
\node at (-0.4,-0.4) {\small $1-\frac{2}{\lambda}$};
\node at (2.7,-0.4) {\small $1-\frac{1}{\lambda}$};
\node at (5, -0.4) {\small $1$};  
\node at (9.5,2.666) {\tiny $2-\lambda$};  
\node at (-0.6,4) {\small $\frac{2(\lambda-1)}{\lambda}$};
\node at (2, 3.5){$T_\lambda$}; 
\node at (0.5, 4.2){$\bar \varphi$};  \node at (3.1, 0.5){$\bar \varphi$}; 
\node at (5.6, 1.5){$B$};  \node at (4.4, 8.5){$B$};  
\end{tikzpicture}
\caption{\label{fig:Outside-Lorenz} Constructing the outside and stunted Lorenz map for a tent map $T_\lambda$.}
\end{center}
\end{figure}

Outside map $B$ was used in \cite{BdCH} to give a complete description of the prime end and accessible 
points structure in unimodal inverse limits embedded in the plane as attractors of an 
orientation-preserving homeomorphism of the plane (or the two-dimensional sphere $\S^2$).

In the rest of the section we restate some results from \cite{BdCH} and relate them to the established 
conjugacy between maps $B$ and $\bar{\varphi}$.

Recall that $I$ denotes the unit interval $[0,1]$. The \emph{inverse limit space with the bonding map $g:I\to I$} 
is a subspace of the Hilbert cube $I^{\N_0}$ defined by
$$
\underleftarrow{\lim}(I, g):=\{(x_0,x_1,x_2,\ldots)\in I^{\N_0}: g(x_{i+1})=x_{i}, \ i\in \N_0 \}.
$$
Equipped with the product topology, the space $\underleftarrow{\lim}(I, g)$ is a {\em continuum}, 
i.e., compact and connected metric space. Define the \emph{shift homeomorphism}
$\widehat{g}:\underleftarrow{\lim}(I,g)\to \underleftarrow{\lim}(I,g)$, 
$$
\widehat{g}((x_0,x_1,\ldots)):=(g(x_0),x_0,x_1,\ldots) \text{ for } 
(x_0,x_1,\ldots)\in \underleftarrow{\lim}(I, g).
$$

There is a natural way to make $\underleftarrow{\lim}(I, g)$ an attractor of an orientation preserving 
sphere homeomorphism. 
Such embeddings are called Barge-Brown-Martin embeddings (abbreviated BBM embeddings), see \cite{BM} 
for the original construction, \cite{BdCH1} for generalisation of the construction to parametrised 
families and \cite{BdCH} for the construction applied to unimodal inverse limits.
As the outcome of the BBM embedding of $\underleftarrow{\lim}(I, g)$, one obtains an orientation-preserving homeomorphism $H\colon\S^2\to \S^2$ so that 
$H|_{\underleftarrow{\lim}(I, g)}$ is topologically conjugate to 
$\widehat{g}$ and for every $x\in \S^2\setminus \{ \text{point}\}$, and $\omega(x,H)$ 
is contained in $\underleftarrow{\lim}(I, g)$.

In \cite{BdCH} the authors study in detail BBM embeddings of inverse limits of unimodal maps satisfying certain (mild) conditions, 
which are in particular satisfied for the tent map family $T_{\lambda}$, $\lambda\in(\sqrt 2, 2]$. 
For simplicity we state the following results for tent maps only, noting that they can be generalized to a much wider class of unimodal maps.

Fix the slope $\lambda\in(\sqrt 2, 2]$ for the tent map $T_{\lambda}$. 
Let $B$ be the corresponding outside map. Denote by $\widehat{S}=\underleftarrow{\lim}\{S^1, B\}$. Theorem 4.28 from \cite{BdCH} shows that there is a natural
homeomorphism $h$ between $\widehat{S}$ and the circle  $P$ of prime ends of $\underleftarrow{\lim}(I, T_{\lambda})$.
 Then $h$ conjugates the shift homeomorphism $\widehat{B}$ of the outside map to the action of $H$ on $P$, 
 so that the prime end rotation number of $\underleftarrow{\lim}(I, T_{\lambda})$ is equal to the 
 rotation number (as defined in \eqref{eq:rot}) of $\widehat{B}$, see \cite[Lemma 4.30]{BdCH}. Finally, Corollary 4.36 in 
 \cite{BdCH} gives that the prime end rotation number of $\widehat{B}$ is equal to the rotation number of $B$. 
Since $\bar{\varphi}$ and $B$ are conjugate, the results above follow analogously and by 
Proposition~\ref{prop:rotnumber} we obtain that the prime end rotation number of $\widehat{B}$ equals $1-\alpha$.

\begin{proposition}\label{prop:primeendrot}
	Let $T_\lambda$ be a tent map with slope $\lambda\in (\sqrt{2},2]$ and let $\bar{\varphi}$ be 
	corresponding stunted Lorenz map with rotation number $\alpha$. 
	Let $\underleftarrow{\lim}([0, 1], T_{\lambda})$ be embedded in $\S^2$ by a BBM construction. 
	Then the prime end rotation number of $\widehat{T}_{\lambda}$ on $\underleftarrow{\lim}([0, 1], T_{\lambda})$ 
	equals $1-\alpha$. 
\end{proposition}

\begin{remark}\label{rem:height}
	In \cite{BdCH}, the prime end rotation number is expressed in terms of the height $q(\nu)$ of the kneading sequence of a unimodal map $f$ (see the definition of height in \eg \cite[Section 2.6]{BdCH}). 
	Proposition~\ref{prop:rotnumber} thus gives an algorithm to compute the height of the kneading sequence in the following way: find the 
	smallest $n\in\N$ such that $c_n\in (1-b, b)$, and $n$ is a cutting time $n=S_k$. 
	Then the height equals $1-k/S_k$. If no such $n$ exists, then the height equals $1-\lim_{k \to \infty} k/S_k$.
	Recall that $b>c$ is such that $f^2(b)=\widehat{f^2(c)}$, so the itinerary of $b$ is 
	$\itin(b)=11\nu_3\nu_4\nu_5\ldots$, see \eqref{eq:ib}, where $\nu=\itin(f(c))=10\nu_3\nu_4\nu_5\ldots$ is the kneading sequence. 
	Hence, the previous condition can be expressed with symbols as $01\nu_3\nu_4\nu_5\ldots\prec\nu_n\nu_{n+1}\nu_{n+2}\ldots\prec 11\nu_3\nu_4\nu_5\ldots$, 
	where $\prec$ denotes the parity-lexicographic ordering on symbolic sequences.
\end{remark}

Furthermore, \cite{BdCH} gives the complete characterisation of accessible points 
the BBM embeddings of $\underleftarrow{\lim}(I, T_{\lambda})$ using the outside map.
We emphasize it here and connect it to the stunted Lorenz map $\overline{\varphi}$.

Let $S\colon[0,2]\to\R^2$ be a circle parametrisation defined by $S(t)=(\frac 12+\frac 12\cos(\pi+t\pi), \frac 12\sin(\pi+t\pi))$ for $t\in[0,2]$. 
Let $\tau\colon S([0,2])\to I$ be the vertical projection, \ie $\tau((x,y))=x$ for $(x,y)\in S([0,2])$. 
Furthermore, let $\mathring{\gamma}=(S(2/\lambda), S(2))$. As before, let $\hat S=\underleftarrow{\lim}\{S([0,2]), B\}$.

\begin{proposition}[Theorem~4.28(d), Remark~4.15, Definition~4.12, Corollary~4.14 in \cite{BdCH}]
	Let $\underleftarrow{\lim}(I, T_{\lambda})$ be embedded in $\R^2$ by an orientation-preserving BBM embedding. Then $(x_0,x_1,x_2, \ldots)\in\underleftarrow{\lim}(I, T_{\lambda})$ is accessible if and only if there exists $N\geq 0$ and $y=(y_0,y_1,y_2, \ldots)\in \hat S$ such that $y_i\not\in\mathring{\gamma}$ for all $i>N$ and such that $x_{N+j}=\tau(y_{N+j})$ for all $j\geq 0$.
\end{proposition}

Using the conjugacy of $B$ and $\bar{\varphi}$, we can state the previous theorem in terms of $\bar{\varphi}$ directly. We parametrise the circle above as $T\colon I\to\R^2$ as $T(t)=(\frac 12+\frac 12\cos(\pi+2t\pi), \frac 12\sin(\pi+2t\pi))$ 
and let $\mathring{\delta}=(T(0), T(\frac{\lambda-1}{\lambda}))$. The vertical projection onto a horizontal diameter is denoted by $\tau$ as above (that is actually $\tau\circ G$, where $G$ is the conjugacy between $\bar{\varphi}$ and $B$ and it is equal to the vertical projection). 
In particular, $\tau\circ\bar{\varphi}(x)=T_{\lambda}\circ\tau(x)$ for all $x\not\in\mathring{\delta}$.

\begin{theorem}\label{thm:accessible}
	Let $T_\lambda$ be a tent map with slope $\lambda\in (\sqrt{2},2]$ and let $\bar{\varphi}$ be 
	corresponding stunted Lorenz map with rotation number $\alpha$. 
	Let $\underleftarrow{\lim}(I, T_{\lambda})$ be embedded in $\S^2$ by an orientation-preserving BBM embedding. 
	Let $\hat S=\underleftarrow{\lim}(S^1,\bar{\varphi})$.
	A point $(x_0,x_1,x_2,\ldots)\in \underleftarrow{\lim}(I, T_{\lambda})$ is accessible if and only if there 
	exists $N\geq 0$ and $y=(y_0,y_1,y_2, \ldots)\in \hat S$ such that $y_i\not\in\mathring{\delta}$ for all $i>N$ 
	and such that $x_{N+j}=\tau(y_{N+j})$ for all $j\geq 0$.
\end{theorem}

We say that a point $x=(x_0,x_1, \ldots)\in\underleftarrow{\lim}(I, T_{\lambda})$ is a {\em folding point} if $x_n\in\omega_{T_{\lambda}}(c)$ for every $n\geq 0$. 
In the context of inverse limits on intervals this is equivalent to saying that $x$ has no neighbourhood homeomorphic to the Cantor set times 
an open interval (see \cite{Raines}). 
In the case when rotation number of $\bar \varphi$ is irrational, Proposition~\ref{prop:rotnumber} and its proof imply that $f\colon\omega_{\varphi}(c)\to\omega_{T_{\lambda}}(c)$ is a homeomorphism (recall that orbits of $c$ 
under $\varphi$ and $\bar{\varphi}$ are the same when the corresponding height of the tent map is irrational). 
From that and Theorem~\ref{thm:accessible} we have the following:

\begin{corollary}\label{cor:fp}
	If $\lambda\in(\sqrt 2, 2]$ and the rotation number of $\bar{\varphi}$ is irrational (\ie the height of the kneading sequence of $T_{\lambda}$ is irrational), then every 
	folding point of $\underleftarrow{\lim}(I, T_{\lambda})$ embedded in $\R^2$ by the orientation-preserving BBM embedding is accessible.
\end{corollary}
\begin{proof}
	We first note that $\tau\colon\omega_{\bar\varphi}(c)\to\omega_{T_\lambda}(c)$ is well defined and bijective. 
	The first part follows since $\tau(\lim_{i\to\infty}\bar{\phi}^{n_i}(c))=\lim_{i\to\infty}\tau\circ\bar\phi^{n_i}(c)=\lim_{i\to\infty}T_{\lambda}^{n_i}\circ\tau(c)=\lim_{i\to\infty}T_{\lambda}^{n_i}(1)\in\omega_{T_\lambda}(c)$, when $\lim_{i\to\infty}\bar\phi^{n_i}(c)$ exists. Similar argument also shows that $\tau$ is surjective. For the proof of injectivity, it is enough to note that $\tau(x)=\tau(y)$ implies that $y=\tilde x$ or $y=x$ and apply the fact that $T_{\lambda}\colon\omega_{\bar\varphi}(c)\to\omega_{T_\lambda}(c)$ is a homeomorphism.
	
	Now let $(x_0, x_1, x_2, \ldots)\in\underleftarrow{\lim}(I, T_{\lambda})$ be such that $x_i\in\omega_{T_{\lambda}(c)}$ 
	for every $i\geq 0$. Then $\bar\phi(\tau^{-1}(x_i))=\tau^{-1}(T_{\lambda}(x_i))=\tau^{-1}(x_{i-1})$ for every $i>0$,
	so $(\tau^{-1}(x_0), \tau^{-1}(x_1), \tau^{-1}(x_2), \ldots)\in\underleftarrow{\lim}(S^1,\bar{\varphi})$. We apply Theorem~\ref{thm:accessible} to conclude that $(x_0, x_1, x_2, \ldots)$ is accessible.
\end{proof}

\begin{remark}	
	Given a continuum $X$, we say that point $x\in X$ is an endpoint if for every subcontinua $A, B\subset X$ such that $x\in A\cap B$, 
	we have $A\subset B$ or $B\subset A$. 
	In \cite{AC} it was shown that if $\underleftarrow{\lim}(I, T_{\lambda})$ is embedded in the plane by an orientation-preserving BBM embedding, and if $\bar{\varphi}$ has irrational rotation number (\ie the height of the kneading sequence of $T_{\lambda}$ is irrational), then all endpoints are accessible.
	Moreover, it was shown that there also exist countably many accessible non-end folding points. 
	Corollary~\ref{cor:fp}\footnote{Its statement was suggested to us by Boyland, de Carvalho, Hall through personal communication.} 
	in particular implies that
	there are uncountably many endpoints and only countably many non-end folding point 
	in $\underleftarrow{\lim}(I, T_{\lambda})$. This partially answers Problem~1 in \cite{AABC}.
\end{remark}

Let us discuss the irrational rotation number case in more details. If $1-\alpha$ (and hence $\alpha$) is irrational, then Proposition~\ref{prop:rotnumber} 
gives that $Q(j) \leq 1$ for all $j$,
and $T_{\lambda}:\omega(c) \to \omega(c)$ is a Cantor minimal homeomorphism conjugate to a Sturmian shift. 
This implies that $\widehat{H}$ induces Denjoy-like dynamics on the corresponding circle of prime ends $P$. 
In \cite{AC}, a detailed characterisation of accessible points for BBM embeddings of tent inverse limit spaces is given (there, 
also accessible endpoints and non-end folding points are distinguished), based solely on symbolic dynamics techniques 
from kneading theory.
It follows (see \cite[Theorem 11.20]{AC}) that there is a Cantor set $C\subset P$ corresponding to accessible folding points
in $\underleftarrow{\lim}(I, T_{\lambda})$ uncountably many of which are endpoints and countably many are non-end folding points. 
Furthermore, all endpoints are accessible. 
The remaining countably infinitely many open arcs 
$P\setminus C$ correspond to countably infinitely many open arcs in different \emph{arc-components} 
of $\underleftarrow{\lim}(I, T_{\lambda})$ (unions of all arcs containing some point from 
$\underleftarrow{\lim}(I, T_{\lambda})$)  that are accessible at more than one point.
Thus this planar continua are interesting also from a topological perspective.
A theorem of Mazurkiewicz~\cite{Maz} 
shows that for every indecomposable planar continuum there are at most countably infinitely many arc-components accessible at more than one point. Our examples confirm that it is possible to find planar continua indeed having countably 
infinitely many arc-components accessible at more than one point. Furthermore,  $T_{\lambda}$ is long-branched
since $\sup_j Q(j) = 1$. Therefore, all proper subcontinua of 
$\underleftarrow{\lim}(I, T_{\lambda})$ are arcs (see \eg \cite[Proposition 3]{BB99}).\\
Thus, from the discussion in this section we have a complete understanding of topology of $\underleftarrow{\lim}(I, T_{\lambda})$ 
as well as their orientation-preserving BBM planar embedding in the case when the rotation number of $\bar{\varphi}$ is irrational (that is, the height of the kneading sequence of $T_{\lambda}$ is irrational).

\end{document}